\documentclass[10pt]{amsart}
\usepackage{amssymb}
\usepackage{latexsym}
\usepackage{amsthm}
\usepackage{amscd}
\usepackage{color}

\newtheorem{thm}{Theorem}[section]
\newtheorem{cor}[thm]{Corollary}
\newtheorem{lem}[thm]{Lemma}
\newtheorem{prop}[thm]{Proposition}

\newtheorem{defn}[thm]{Definition}

\numberwithin{equation}{section}

\newcommand{\cj}{\overline}

\newcommand{\al}{\alpha}
\newcommand{\Ga}{\Gamma}
\newcommand{\ep}{\varepsilon}

\allowdisplaybreaks[1]

\newcommand{\qq}{{q_{\circ}}}
\newcommand{\rr}{{r_{\circ}}}

\begin{document}
\title[mass critical fractional  Schr\"odinger
equation]{Profile decompositions and Blowup phenomena  of \\ mass critical fractional
Schr\"odinger equations}
\author[Y. Cho]{Yonggeun Cho}
\address{Department of Mathematics, and Institute of Pure and Applied Mathematics, Chonbuk National University, Jeonju 561-756, Republic of Korea}
\email{changocho@jbnu.ac.kr}
\author[G. Hwang]{Gyeongha Hwang}
\address{Department of Mathematics, Pohang University of Science and Technology, Pohang 790-784, Republic of Korea}
\email{g\_hwang@postech.ac.kr}
\author[S. Kwon]{Soonsik Kwon}
\address{Department of Mathematical Sciences, Korea Advanced Institute of Science and Technology, Daejeon 305-701, Republic of Korea} \email{soonsikk@kaist.edu}
\author[S. Lee]{Sanghyuk Lee}
\address{Department of Mathematical Sciences, Seoul National University, Seoul 151-747, Republic of Korea}
\email{shklee@snu.ac.kr}

\thanks{}

\begin{abstract}
We study, under the radial symmetry assumption, the solutions to the
fractional Schr\"odinger equations of  critical nonlinearity in
$\mathbb R^{1+d}, d \geq 2$, with L\'{e}vy index ${2d}/({2d-1}) < \al < 2$. We
firstly prove the linear profile decomposition 
and then apply it to investigate the properties of the blowup
solutions of the nonlinear equations with mass-critical Hartree type nonlineartity.
\end{abstract}

\maketitle
\section{Introduction}

In \cite{la1} Laskin introduced the fractional quantum mechanics in
which he generalized the Brownian-like quantum mechanical path, in
the Feynman path integral approach to quantum mechanics, to the
$\alpha$-stable L\'evy-like quantum mechanical path.
This gives a rise to the fractional generalization of the
Schr\"odinger equation.  Namely, the associated equation for the
wave function results in the fractional Schr\"odingier equations,
which contains a nonlocal fractional derivative operator
$(-\Delta)^\frac\alpha2$ defined by $(-\Delta)^\frac\alpha2=\mathcal
F^{-1}|\xi|^\alpha \mathcal F$. In this paper we consider the
following Cauchy problem with mass critical Hartree type
nonlinearity:
\begin{align}\label{eqn}
\left\{\begin{array}{l} iu_t + (-\Delta)^\frac\alpha2 u = \lambda (|x|^{-\al} * |u|^2)u, \;\;(t, x) \in \mathbb{R}^{1+d},\,\, d \ge 2,\\
u(0, x) = f(x) \in L^2,
\end{array}\right.
\end{align}
where $\lambda=\pm1$. Here $\alpha$ is L\'{e}vy stability index with
$1 < \alpha \le 2$. When $\alpha = 2$, the fractional Schr\"odinger
equation becomes the well-known Schr\"odinger equation. See \cite{la2,
la3} and references therein for further discussions related to the
factional quantum mechanics.

The solutions to equation \eqref{eqn} have the conservation laws for
the mass and the energy:
\begin{align*}
M(u) &= \int |u|^2 \,dx,\,\, E(u) = \frac 12 \int
\cj{u}|\nabla|^\alpha u \,dx - \frac\lambda4 \int \cj{u}(|x|^{-\al} * |u|^2)u\,dx.
\end{align*}
We say  that \eqref{eqn} is focusing if $\lambda =1$, and defocusing
if $\lambda=-1$. The equation \eqref{eqn} is mass-critical, as
$M(u)$ is invariant under scaling symmetry $  u_\rho (t,x) =
\rho^{-d/2}u({t}/{\rho^{\alpha}},$ ${x}/\rho )$, $\rho>0$ which is
again a solution to \eqref{eqn} with initial datum $\rho^{-d/2}u(0,
{x}/\rho)$.  The equation \eqref{eqn} is locally well-posed in $L^2$
for radial initial data and globally well-posed for sufficiently
small radial data \cite{chho}. (See \cite{guwa} for results
regarding power type nonlinearities.) For the focusing case,  the
authors \cite{chkl} used a virial argument to show the finite time
blowup, with radial data,
provided that the energy $E(u)$ is negative. Also see \cite{guhu}
and \cite{guhaxi} for results with noncritical nonlinearity.

In this paper we aim to investigate the blowup phenomena of
\eqref{eqn} with radial data when $\alpha<2$. Due to the critical
nonlinearity the time of existence no longer depends on the $L^2$
norm of initial data. Instead it relies on the profiles of the data.
Hence the situation become more subtle.   When
$\alpha=2$, a lot of work was devoted to the study of blowup
phenomena, which was based on the usual Strichartz and its
refinements. (See for instance \cite{keme,ker2,meve}.)
When it comes to the fractional the Schr\"odinger equation
$(1<\alpha<2)$, due to the non-locality of fractional operator,
various useful properties (e.g. Galilean invariance) which hold for the Schr\"odinger equation
are no longer available. The main difficulty comes from absence of
proper linear estimates. In fact, by scaling the condition
$\alpha/q+d/r=d/2$ should be satisfied by the pair $(q,r)$ if
$L^2$-$L_t^qL_x^r$ estimates were true for the linear propagator
$f\to e^{it(-\Delta)^{\frac\alpha2}}f$. But such estimate is
impossible as the Knapp example shows that $\dot{H}^s$-$L_t^qL_x^r$ is
only possible for $2/q + d/r \le d/2$. In order to get around these
difficulties we work with radial assumption on the initial data,
which allows us to use the recent results on the Strichartz
estimates for radial functions \cite{guwa} or angularly regular
functions \cite{cholee}.

\subsubsection*{\textbf{Linear profile decomposition}}
As for linear estimates such as the Strichartz estimates or Sobolev
inequalities, the presence of noncompact symmetries causes defect of
compactness. The profile decomposition with respect to the
associated linear estimates is a measure to make it rigorous that
such symmetries are the only source of non-compactness.

Concerning nonlinear  dispersive equations (especially nonlinear
wave and Schr\"o-dinger equations), the profile decompositions have
been intensively studied and led to various recent developments in
the study of  equations with the critical nonlinearity
(\cite{keme}). Profile decompositions for the Schr\"odinger
equations with $L^2$ data were obtained by Merle and Vega
\cite{meve} when $d=2$, Carles and Keraani \cite{ck}, $d=1$, and
B\'egout and Vargas [3], $d\ge 3$. ( Also see \cite{bage, bul,
ram} for results on the wave equation and \cite{sh, ksv, favi} on
general dispersive equations.) These results are based on
refinements of Schrichartz estimates (see \cite{mvv, bo1}). There is
a different approach based on Sobolev imbedding but such approach is
not applicable especially the equation is $L^2$-critical. Our
approach is also based on  a refinements of Schrichartz estimate.
Thanks  to the extended range of admissible due to the radial
assumption it is relatively simpler to obtain the refinement (see
Proposition \ref{ref-str-prop} which is used for the proof of
profile decomposition.)

 We now define the linear propagator $U
(t)f$ to be the solution to the linear equation $iu_t +
(-\Delta)^\frac\alpha2 u = 0$ with initial datum $f$. Then it is
formally given  by
\begin{align}\label{int eqn}
U (t)f = e^{it(-\Delta)^{\frac\alpha2}}f=\frac1{(2\pi)^d}
\int_{\mathbb{R}^d} e^{i(x\cdot \xi + t
|\xi|^\alpha)}\widehat{f}(\xi)\,d\xi.
\end{align}
Here $\widehat{f}$ denotes the Fourier transform of $f$ such that
$\widehat{f}(\xi) = \int_{\mathbb{R}^d} e^{-ix\cdot \xi} f(x)\,dx$.

The following is our first result:

\begin{thm}\label{main}
Let  $d \ge 2$, $\frac{2d}{2d-1} < \alpha < 2$, and $2 < q, r <
\infty$ satisfy $\frac \al q + \frac dr = \frac d2$.  Suppose that
$(u_n)_{n \geq 1}$ is a sequence of complex-valued radial functions
satisfying $\|u_n\|_{L^2_x} \leq 1$. Then up to a subsequence, for
any $l \geq 1$, there exist a sequence of radial functions
$(\phi^j)_{1\leq j \leq l}\in L^2$, $\omega_n^l \in L^2$ and a
family of parameters $(h_n^j, t_n^j)_{1 \leq j \leq l, n \geq 1}$
such that
$$u_n(x) = \sum_{1 \leq j \leq l} U(t^j_n)[(h^j_n)^{-d/2}\phi^j(\cdot/{h^j_n})](x) + \omega^l_n(x)$$
and the following properties are satisfied:
$$\lim_{l \rightarrow \infty} \limsup_{n \rightarrow \infty}
\|U_\alpha(\cdot)
\omega^l_n\|_{L^q_tL^r_x (\mathbb{R} \times
\mathbb{R}^d)} = 0,$$ and for $j \neq k$, $(h^j_n, t^j_n)_{n \geq
1}$ and $(h^k_n, t^k_n)_{n \geq 1}$ are asymptotically orthogonal in
the sense that
\begin{align*}
&\mbox{either}\;\; \limsup_{n \rightarrow
\infty}\left(\frac{h^j_n}{h^k_n} +
\frac{h^k_n}{h^j_n}\right) = \infty,\\
&\mbox{or}\;\; (h^j_n) = (h^k_n)\;\;\mbox{and}\;\;\limsup_{n
\rightarrow \infty}\frac{|t^j_n - t^k_n|}{(h^j_n)^\alpha} = \infty,
\end{align*}
and for each $l$
$$\lim_{n \rightarrow \infty} \Big[\|u_n\|_{L^2_x}^2 - (\sum_{1 \leq j \leq l}\|\phi^j\|_{L^2_x}^2
+ \|\omega^l_n\|_{L^2}^2) \Big] = 0.$$
\end{thm}

In what follows we make use of the linear profile decomposition to
get nonlinear profile decompositions of the solutions to
\eqref{eqn}.

\subsubsection*{\textbf{Nonlinear profile decomposition}}
Let us set
\[(q_\circ,r_\circ)=\Big(3,\frac
{6d}{3d - 2\al}\Big).\] As it can be shown by the usual fixed point
argument and the Strichartz estimate (with $\alpha$-admissible
pairs), in Lemma \ref{str-radial} the local well-posedness
theory can be based on the estimate of space-time norm  $
\|u\|_{L_t^\qq L_x^\rr (I \times\mathbb{R}^d)}$ \footnote{In fact,
one use choose any $\alpha$-admissible $(q,r)$ such that
$6d/(3d-\alpha)\le r\le 6d/(3d-2\alpha)$ if  $2d/(2d-1)<\alpha<2$
and $6d/(3d-\alpha)< r\le 6d/(3d-2\alpha)$ if  $2d/(2d-1)=\alpha$.
For instance see \cite{chl}.}. As a by product, if the solution
fails to persist, then the space-time norm blows up.


\begin{defn}
A solution $u \in C_tL^2_x((-T_{min},T_{max})\times\mathbb{R}^d)$ to
\eqref{eqn} is said to blow up if $\|u\|_{L_t^\qq
L_x^\rr((-T_{min},T_{max})\times\mathbb{R}^d)} = \infty$. Here
$-T_{min}, T_{max}\in [-\infty,\infty]$ denote the maximal times
of existence of the solution.
\end{defn}
Since $ T_{max}$ or $T_{min}$ may be $\infty$, we regard
non-scattering global solutions as blowup solutions at
infinite time.
 We also define the minimal mass of solutions from which a solution may ignite to blow up.
\begin{defn}\label{minimal mass}
$\delta_0 := \sup\, \{A \ge 0 : $ If $ \|u_0\|_{L^2_x} < A$, for all
$u_0 \in L^2_x$ \eqref{eqn} is globally well-posed forward and
backward,  and its solution $u$ satisfies $\|u\|_{L_t^\qq
L_x^\rr((-\infty,\infty)\times\mathbb{R}^d)}$
$< \infty\}$.
\end{defn}
By the small data global existence, we have $\delta_0 >0$. Moreover,
for any $ \delta >\delta_0 $, there exists a blowup solution $u$
with $\delta_0 \le \|u\|_{L^2} \le \delta $. In
Theorem~\ref{th:minimal blow up} below we show that there exists a blowup
solution having the minimal mass $\delta_0$. This will be shown by
using the nonlinear profile decomposition, which is derived from the
linear profile decomposition combined with perturbation theory.

For a given sequence of radial data $ \{u^0_n\} \subset L^2_x $,
from the linear profile decomposition (Theorem~\ref{main}), we have
an asymptotically orthogonal decomposition to a sequence
$\{\phi^j\}_{1\leq j \leq l} \in L^2$, $\omega_n^l \in L^2$,
$(h_n^j, t_n^j)_{1 \leq j \leq l, n \geq 1}$. Then by taking
subsequence, if necessary, we may assume that $ t_n \in \{-\infty,
0, \infty\}$. Here we denote $t_n = \lim_j t^j_n $. Using the
local well-posedness theorem with initial data at $t=0$ or $t=\pm
\infty $ (see Lemma~\ref{wpasy} below), we define the nonlinear
profile by the maximal nonlinear solution for each linear profile.


\begin{defn}
Let $\{h_n,t_n\} $ be a family of parameters and $\{t_n\}$
have a limit in $[-\infty, \infty]$. Given a linear profile $\phi
\in L^2_x$ with $\{(h_n,t_n)\}$, we define the nonlinear profile
associated with it to be  the maximal solution $\psi$ to \eqref{eqn}
which is in $C_tL^2_x((-T_{\min},T_{\max}) \times \mathbb{R}^d)$
satisfying an asymptotic condition: For the sequence $\{t_n\}$,
$$\lim_{n \to \infty}\|U(t_n)\phi - \psi(t_n)\|_{L^2_x} = 0.$$
\end{defn}

\noindent Then, the linear profile decomposition yields the
nonlinear profile decomposition which  is the key tool for proving
blowup phenomena in what follows.
\begin{prop}\label{nlpf}
Let $\{u_n^0\} \subset L^2_x$ be a bounded sequence. Suppose that
$\{\phi^j\}_{1\leq j \leq l} \in L^2$, $\omega_n^l \in L^2$,
$(h_n^j, t_n^j)_{1 \leq j \leq l, n \geq 1}$ is a linear profile
decomposition obtained from Theorem~\ref{main}. Let $u_n \in
C(J_n;L^2_x)$ be the maximal solution of \eqref{eqn} with initial
data $u_n(t_0) = u_n^0$. For each $j\ge 1$, suppose
$\{\psi^j\}_{1\leq j \leq l} \in C_tL^2_x((-T^j_{\min},T^j_{\max})
\times \mathbb{R}^d)$ is the maximal nonlinear profile associated
with $\{\phi^j\}_{1\leq j \leq l}, (h_n^j, t_n^j)_{1 \leq j \leq l,
n \geq 1}$. Let $\{I_n\}$ be a family of nondecreasing time intervals containing
$0$. Then, the following two are equivalent;
\begin{enumerate}
\item $\| \Ga^j_n \psi^j \|_{L_t^\qq L_x^\rr(I_n \times \mathbb{R}^d)} < \infty, \quad j\ge1,$
\item $\| u_n \|_{L_t^\qq L_x^\rr(I_n \times \mathbb{R}^d)} < \infty.$
\end{enumerate}
Here  $\Gamma^j_n\psi^j = \frac 1{(h^j_n)^{d/2}}\psi^j(\frac {t -
t^j_n}{(h^j_n)^\al},\frac x{h^j_n})$. Moreover, if $(1)$ or $(2)$
holds true, we have a decomposition
$$u_n = \sum_{j=1}^l \Ga^j_n \psi^j + U(\cdot)\omega^l_n + e^l_n$$
with
$$\lim_{l \to \infty}\limsup_{n \to \infty} \big(\|\big(U(\cdot)\omega^l_n\|_{L_t^\qq L_x^\rr(I_n \times \mathbb{R}^d)}
+ \|e^l_n\|_{L_t^\qq L_x^\rr(I_n \times \mathbb{R}^d)} \big)= 0.$$
\end{prop}

\subsubsection*{\textbf{Blowup phenomena}} We now consider the blowup solutions of
\eqref{eqn} and present various results which rely on the nonlinear
profile decomposition.

We first show the existence of minimal mass blowup solution. Due to
lack of compactness of the Strichartz estimate, we do not expect
that a bounded sequence has a convergent subsequence. However, the
extremal sequence has a convergent subsequence and its limit. This
can be viewed as a Palais-Smale type theorem (see \cite{lion}).
\begin{thm}\label{th:minimal blow up}
There exists a blowup solution $u$ to $\eqref{eqn}$ with initial
data $f \in L^2$ of $\|f\|_{L^2} = \delta_0$. Moreover, $\{u(t) \in
L^2 :  -T_{min} < t < T_{max}  \} $ is compact in $L^2$ modulo
symmetries. That is, for any sequence $\{u(t_n)\}$ with $t_n \in
(-T_{min},T_{max} )$, there exist $\phi \in L^2$ and a subsequence,
still called $\{t_n\}$ and $\{h_n\}$, such that
$$   h_n^{d/2}u(t_n,h_nx) \to \phi \quad {in} \,\, L^2. $$
\end{thm}

If the mass is greater than the minimal mass($=\delta_0^2$) but less
than twice of $\delta_0^2$, then the blowup solution does not
form more than one blowup profile. Thus, we still have a weaker form
of compactness property of the blowup solutions.
\begin{thm}\label{cpbs}
Let $u$ be finite time blowup solution of \eqref{eqn} at $T^*$ with
$\|u(0,\cdot)\|_{L^2_x} < \sqrt{2}\delta_0$ and let  $t_n
\nearrow T^*$. Then there exist $\phi \in L^2_x$ and
$\{h_n\}_{n=1}^\infty$ satisfying \begin{equation} \label{weakcon}
h^{d/2}_n u(t_n, h_nx) \rightharpoonup \phi \text { weakly in }
L^2_x\end{equation}
 and if solution of \eqref{eqn} with initial data $\phi$ blows up at
$T^{**}$
\begin{equation}
\label{hhh} \limsup_{n \rightarrow \infty} \frac {h_n}{(T^* -
t_n)^{1/\alpha}} \leq \frac 1{(T^{**})^{1/\alpha}}. \end{equation}
\end{thm}

Under the same condition, when a blowup occurs, only one profile
blows up by shrinking in scale. As a corollary, we obtain the
concentration in $L^2$-norm at blowup time. For related results
when $\al > 2$, see \cite{chl}. More precisely, we have

\begin{cor}{(Mass concentration of finite time blowup solution)}\label{mass concentration}
Let $u$ be a finite time blowup solution at $T^*$ with
$\|u(0,x)\|_{L^2} < \sqrt{2}\delta_0$ and let  $t_n \nearrow
T^*$. Then
\begin{equation}\label{masscon}
\limsup_{n \rightarrow \infty} \int_{|x| \leq \lambda(t_n)}
|u(t_n,x)|^2 dx \geq \delta_0^2 \end{equation} for $\lambda(t_n)$
satisfying $\frac {(T^* - t_n)^{1/\alpha}}{\lambda(t_n)} \rightarrow
0$.
\end{cor}

The arguments in this paper can be modified to prove the same
results (nonlinear profile decomposition and blowup phenomena) for
the equations which have the power-type mass critical nonlinearities
as long as we assume that blowup occurs. But the existence of blowup solutions does not seem to be known yet for the fractional
Schr\"odinger equations with power-type nonlinearities.



The rest of the paper is organized as follows: In Section 2, we will
show the refined Strichartz estimate. Section 3 will be devoted to
establishing linear profile decomposition. Then we will show the nonlinear profile decomposition in Section 4. In Section 5 we will study blowup phenomena by making use of the profile decomposition.


%
%
%
%
%
%
%
%

%
%
%

\section{Refined Strichartz Estimates}\label{sec2}
It has been known that the Strichartz estimates for dispersive
equations have  wider admissible ranges  when the initial data $f$
are radial \cite{guwa, cholee}. Recently,  almost optimal range of
admissible pairs was established in \cite{guwa} and the range was
further extended in \cite{cholee, ke} to
include the remaining endpoint cases.  We now recall from
\cite{cholee} that
\begin{equation}\label{basic}
\| U(\cdot)P_0 f \|_{L^q_tL^r_x} \lesssim \|f\|_2 \end{equation}
holds whenever $q,r\ge 2$,  $(q, r) \neq (2, \frac{2(2d-1)}{2d-3})$,
and $\frac1q \le \frac{2d-1}{2}(\frac12 - \frac1r)$.

Let $P_k, k \in \mathbb Z,$ denote the Littlewood-Paley
projection operator with symbol $\chi(\xi/2^k) \in C_0^\infty$
supported in the annulus $A_k = \{2^{k-1} < |\xi| \le
2^{k+1}\}$ such that $\sum_{k\in \mathbb Z} P_k=id$. By \eqref{basic}, Littlewood-Paley decomposition and rescaling we get
the following.

\begin{lem}\label{str-radial}
Let $\frac{2d}{2d-1} < \al < 2$, $q,r \geq 2,$ and $ r \neq \infty$,
and let $\beta(\al,q,r) = d/2 - d/r - \al /q$.
 If $(q, r) \neq (2, \frac{2(2d-1)}{2d-3})$ and $\frac 1q \leq
\frac{2d-1}{2}(\frac 12 - \frac 1r)$, then for radial $f$,
$$\| U(\cdot) f \|_{L^q_tL^r_x} \lesssim \Big(\sum_{k \in \mathbb Z} 2^{2k\beta(\al,q,r)}\|P_kf\|^2_2\Big)^{1/2}.
$$\end{lem}

Adapting the argument of \cite{chl} together with Lemma
\ref{str-radial}, we get bilinear  estimates for $U$ which give
extra smoothing due to interaction of two waves at different
frequency levels.
\begin{lem}\label{bilinear-est}
Let $\frac{2d}{2d-1} < \al < 2$, $q,r > 2,$ and $ r \neq \infty$.
Suppose that $f$ and $g$ are radial. Then, for $\frac 1q <
\frac{2d-1}{2}(\frac 12 - \frac 1r)$, then there exists $\epsilon =
\epsilon(\al, q, r) > 0$ such that
$$\|U(\cdot)P_jf \;U(\cdot)P_kg\|_{L^{q/2}_tL^{r/2}_x} \lesssim 2^{(j+k)\beta(\al,q,r)}2^{|j-k|\epsilon}\|f\|_{L^2}\|g\|_{L^2}.$$
\end{lem}

\begin{proof} By symmetry we may assume that $k\ge j$ and we set $\ell=j-k \le
0$. Then, by rescaling it suffices to show, for some $\epsilon>0$,
\begin{align}\label{pre-est}
\|U(\cdot)P_{\ell }f\; U(\cdot)P_0g\|_{L^{\frac q2}_tL^{\frac r2}_x}
\lesssim 2^{\ell(\beta(\al, q,r)+ \epsilon)}\|f\|_2\|g\|_2.
\end{align}
This estimate \eqref{pre-est} with $\epsilon=0$  obviously holds for
$\frac1q \le \frac{2d-1}{2}(\frac12 - \frac1r)$, which follows from
Lemma \ref{str-radial} and H\"{o}lder inequality. We can then
interpolate this with the estimate
\begin{align}\label{reduction}
\|U(\cdot)P_{\ell}f\;U(\cdot)P_0g\|_{L^2_{t,x}} &\lesssim
2^{\ell(\beta(\al,4,4) + \epsilon)}\|f\|_2\|g\|_2
\end{align}
for some $\epsilon > 0$ to get \eqref{pre-est} for $\frac 1q <
\frac{2d-1}{2}(\frac 12 - \frac 1r)$. Hence we reduce to showing
\eqref{reduction}.

\medskip

 When $2^\ell \sim 1$  \eqref{reduction} is trivial from
Lemma \ref{str-radial}. Thus we assume $2^\ell \ll 1$. By finite
decomposition, rotation and a mild dilation, we also may assume that
$\widehat g$ is supported $\subset B(e_1, \epsilon)$. Here $e_1 =
(1,0,\cdots,0)$ and $B(e_1, \epsilon)$ is the ball of radius
$\epsilon$ centered at $e_1$. %
 Freezing $\bar \xi=(\xi_2,\dots, \xi_d)\in B(0,1)$,
 we set
 \[ B_{\bar \xi} (f,g)(x,t)=\frac1{(2\pi)^{2d}}\int e^{i(\xi + \eta)x + it(|\xi|^\al
+ |\eta|^\al)}\widehat{P_\ell f}(\xi)\widehat{P_0g}(\eta) d\xi_1
d\eta.\] Then it follows that \begin{equation}\label{bibi}
U(t)P_\ell f \;U(t)P_0g= \int B_{\bar \xi} (f,g)(x,t)
\,d\bar\xi.\end{equation} We make change of variables $(\xi_1, \eta)
\to \zeta= (\xi_1 + \eta_1, \cdots, \xi_d + \eta_d, |\xi|^\al +
|\eta|^\al)$ for $B_{\bar \xi} (f,g)(x,t)$. Then, by  noting
$\Big|\frac {\partial{(\zeta_1, \cdots,
\zeta_{d+1})}}{\partial{(\eta_1, \cdots, \eta_d, \xi_1})}\Big| =
\alpha\big|\,\xi_1|\xi|^{\al -2} - \eta_1|\eta|^{\al - 2}\big| \sim
1,$  Plancherel's theorem, and reversing change of variables (
$\zeta\to (\xi_1, \eta)$ ), we get
\[ \|B_{\bar \xi} (f,g)\|_{L^2_{t,x}}\le C\|\widehat{P_\ell f}(\xi)\widehat{P_0g}(\eta)\|_{L^2_{\xi_1,\eta}}.\]
Therefore, by \eqref{bibi}, Minkowski's inequality, and H\"older's
inequality we get
\begin{align*}
\|U(\cdot)P_\ell fU(\cdot)P_0g\|_{L^2_{t,x}}
&\lesssim \int \|\widehat{P_\ell f}(\xi)\widehat{P_0g}(\eta)
\|_{L^2_{\xi_1,\eta}}d\bar\xi\\
&\lesssim \|g\|_2 \int \|\widehat{P_\ell
f}(\xi)\|_{L^2_{\xi_1}}\chi_{\{|\bar\xi| < 2^{\ell+1}\}} d\bar\xi
&\lesssim 2^{\ell (\frac {d-1}2)}\|f\|_2\|g\|_2.
\end{align*}
Since $\al \ge 0$ and $d \ge 2$, $\beta(\al,4,4) < \frac{d-1}2$.
Hence, we can take  $\epsilon=\frac{d-1}2 - \beta(\al,4,4)$. This
completes the proof of Lemma \ref{bilinear-est}.
\end{proof}

The estimate in Lemma \ref{str-radial} can be strengthened to get
so-called \emph{refinements of Stirchartz estimate} (\cite{bo1, mvv, mvv1}). It
plays crucial role in the proof of profile decomposition. Thanks to
radial symmetry,   such refinement is much easier to obtain.  Here
we make use of the argument in \cite{chl} where high order cases
($\alpha>2$) were treated.

For $\alpha<2$,  let us call the pair $(q,r)$ $\al-$admissible,
provided that $\frac \al q + \frac dr = \frac d2$ for $2 \le q, r
\le \infty$.

\begin{prop}\label{ref-str-prop}
Let $\frac{2d}{2d-1} < \al < 2$, $q > 2$ and $r \neq \infty$. If
$(q,r)$ be $\al-$admissible, then there exist $\theta, p$,  $\theta
\in (0,1), 1 \leq p < 2$, such that
\begin{align}\label{ref-str}
\|U(\cdot)f\|_{L^q_tL^r_x} \lesssim \big(\sup_k 2^{kd(\frac 12 -
\frac 1p)} \|\widehat{P_kf}\|_p\big)^\theta \|f\|_2^{1-\theta}.
\end{align}
\end{prop}

\begin{proof}
We have from Lemma \ref{str-radial}
\begin{align}\label{ref-1}
\|U(\cdot)f\|_{L^q_tL^r_x} \lesssim \Big(\sum_k
\|\widehat{P_kf}\|_2^2\Big)^{1/2}
\end{align}
for any $\al$-admissible pair $(q, r)$. Then \eqref{ref-str} follows
from interpolation of \eqref{ref-1} and the following two estimates:
for some $p_*,q_*$ with $p_* < 2 < q_*$,
\begin{align}
\|U(\cdot)f\|_{L^q_tL^r_x} &\lesssim \Big(\sum_k \| \widehat{P_kf} \|_2^{q_*}\Big)^{1/{q_*}},\label{ref-2}\\
\|U(\cdot)f\|_{L^q_tL^r_x} &\lesssim \Big(\sum_k \big(2^{kd(\frac 12
- \frac 1{p_*}} \| \widehat{P_kf}
\|_{p_*}\big)^2\,\Big)^{1/2}.\label{ref-3}
\end{align}
In fact, the interpolation among \eqref{ref-1}, \eqref{ref-2} and
\eqref{ref-3} gives
\begin{equation}\label{intpol}
\|U(\cdot)f\|_{L^q_tL^r_x} \lesssim \Big(\sum_k\big(2^{kd(\frac 12 -
\frac 1{p_0})} \| \widehat{P_kf} \|_{p_0}\big)^{q_0}\Big)^{ 1/{q_0}}
\end{equation} for $(1/{q_0},1/{p_0})$ on the triangle
with the vertices $(1/2,1/2)$, $(1/{p_*}, 1/2)$ and $(1/2,
1/{q_*})$. So, there exist $q_0,p_0$, $p_0 < 2 < q_0$, for which
\eqref{intpol} holds. Hence,
\begin{align*}
\|U(\cdot)f\|_{L^q_tL^r_x} &\leq \Big(\Big(\sup_k 2^{kd(\frac 12 -
\frac 1{p_0})} \|\widehat{P_kf}\|_{p_0}\Big)^{q_0-2} \sum_k
\big(2^{kd(\frac 12 - \frac 1{p_0})} \| \widehat{P_kf}
\|_{p_0}\big)^{2}\Big)^{1/{q_0}}
\\
&\leq \Big(\sup_k 2^{kd(\frac 12 - 1/{p_0})}
\|\widehat{P_kf}\|_{p_0}\Big)^{(q_0-2)/q_0}\Big(
\sum_k \| \widehat{P_kf} \|_{2}^{2}\Big)^{1/{q_0}}\\
&\leq \Big(\sup_k 2^{kd(\frac 12 - \frac 1{p_0})}
\|\widehat{P_kf}\|_{p_0}\Big)^{(q_0-2)/q_0} \|f\|_2^{2/q_0}.
\end{align*}
For the second inequality we used H\"older's inequality. We need
only to set $p=p_0$ and $\theta=1-2/q_0$ to get \eqref{ref-str}. Now
we need to show \eqref{ref-2} and  \eqref{ref-3}.

We show \eqref{ref-3} first.  Let $(q,r)$, $2<q<\infty$ be
$\alpha$-admissible. Since $\frac {2d}{2d-1} < \al$, there exist $2<
q_0,r_0<\infty$ such that $\frac n2 - \frac n{r_0} - \frac \al{q_0}
> 0$,  $\frac1{q_0} \le \frac{2d-1}{2}(\frac12 - \frac1{r_0})$, and
$(\frac1q,\frac1r)=\theta(\frac1{q_0},\frac1{r_0})$, $0<\theta<1$.
So, by \eqref{basic} we have $ \|U(\cdot) P_0f
\|_{L^{{q_0}}_tL^{{r_0}}_x} \lesssim \|\widehat{P_0f}\|_2. $ By
interpolation of this with the trivial estimate $\|U(\cdot)
P_0f\|_{L^\infty_{t,x}} \lesssim \|\widehat{P_0f}\|_1$, we get  for
some $1<p_*<2$, $\|U(\cdot) P_0f\|_{L^q_tL^r_x} \lesssim
\|\widehat{P_0f}\|_{p_*}.$ Then by rescaling, we have $\|U(\cdot)
P_kf\|_{L^q_tL^r_x} \lesssim 2^{kd(\frac 12 - \frac
1{p_*})}\|\widehat{P_0f}\|_{p_*}.$ Now Littlewood-Paley and
Minkowski inequalities give
\begin{align*}
\|U(\cdot)f\|_{L^q_tL^r_x} \le \Big(\sum_k \|U(\cdot)
P_kf\|_{L^q_tL^r_x}^2\Big)^{1/2} \lesssim \Big(\sum_k
\big(2^{kd(\frac 12 - \frac
1{p_*})}\|\widehat{P_kf}\|_{p_*}\big)^2\Big)^{1/2}.
\end{align*}

We now turn to the proof of \eqref{ref-2}. It is sufficient to show
an $L_{t,x}^4$ estimate
\begin{align}\label{l4}\|U(\cdot)f\|_{L^4_{t,x}} \lesssim \Big(\sum_k(2^{k\beta(\al, 4,4)}
\|\widehat{P_kf}\|_2)^4\Big)^{\frac 14}.\end{align} Indeed, as
before, the required estimates can be obtained by interpolating
\eqref{l4} with the estimates in Lemma
\ref{str-radial} for $(q, r) \neq (2, \frac{2(2d-1)}{2d-3})$
which satisfy $\frac1\al(\frac12-\frac1r) < \frac1q \le
\frac{2d-1}2\Big(\frac12-\frac1r\Big)$, $2 \le q, r \le \infty$.
To show \eqref{l4} we write
\begin{align*}
U(t)f\,U(t)f= \sum_{j =-\infty}^\infty \sum_k
U(t)P_kf\,U(t)P_{j+k}f.
\end{align*}
For \eqref{l4} it is sufficient to show that  for some $\epsilon>0$
\begin{align}\label{l2-half}
\|\sum_k U(\cdot) P_kf\, U(\cdot) P_{j+k}f\|_{L^2_{t,x}} \lesssim
2^{-|j|\epsilon} \Big(\sum_k \big(2^{k\beta(\al,
4,4)}\|\widehat{P_kf}\|_2\big)^4\Big)^{1/4}.
\end{align}

We show it by considering the cases $|j|\le 3$ and  $|j| > 3$,
separately. Let us first consider the case $|j|\le 3$. By
Cauchy-Schwartz inequality, $ |\sum_k U(t) P_kf \,U(t) P_{j+k}f|^2
\leq \sum_{l=-\infty}^\infty \sum_k |U(t)P_kf\,U(t)P_{k+l}f|^2. $
So,
\[  \|\sum_k U(t)P_k f\, U(t)P_{j+k}f\|_{L^2_{t,x}}^2
 \le \sum_{l=-\infty}^\infty \sum_k \| U(t)P_k f\,
 U(t)P_{k+l}f\|_{L^2_{t,x}}.\]
From Lemma \ref{bilinear-est} and Cauchy-Schwartz inequality, it
follows that
\begin{align*}
  \|\sum_k U(t)P_k f\,& U(t)P_{j+k}f\|_{L^2_{t,x}}^2
  \lesssim \sum_{l=-\infty}^\infty\sum_k 2^{-\epsilon
|l|}2^{2k\beta(\al,4,4)}
 \|P_kf\|_2^2 \,2^{2(k+l)\beta(\al,4,4)}\|P_{k+l}f\|_2^2 \\
&\leq \sum_{l=-\infty}^\infty 2^{-\epsilon |l|}\sum_k
\Big(2^{k\beta(\al,4,4)}\|P_kf\|_2^4\Big)^4\lesssim \sum_k
\Big(2^{k\beta(\al,4,4)}\|P_kf\|_2^4\Big)^4
\end{align*}
for $|j|\le 3$. We now consider the case $|j| > 3$. Let us first
observe that the Fourier supports of $U(t)P_kf\,U(t)P_{k+j}f$  are
boundedly overlapping. So by Plancherel's theorem and Lemma
\ref{bilinear-est}
\begin{align*}
&\qquad\qquad\|\sum_k U(\cdot) P_kf\, U(\cdot)P_{k+j}f
\|_{L^2_{t,x}}^2
\lesssim \sum_k \|U(\cdot)P_kf\, U(\cdot)P_{k+j}f \|_{L^2_{t,x}}^2 \lesssim \\
&\sum_k
2^{-|j|\epsilon}2^{k\beta(\al,4,4)}\|P_kf\|_2^22^{(k+l)\beta(\al,4,4)}\|P_{k+l}f\|_2^2
\lesssim 2^{-|j|\epsilon} \sum_k (2^{k\beta(\al,4,4)}\|P_kf\|_2)^4.
\end{align*}
This completes the proof.
\end{proof}

\section{Linear profile decomposition}
In this section we prove Theorem \ref{main}. We assume that the pair
$(q ,r)$ is $\alpha$-admissible with $\frac{d}2(\frac12 - \frac1r) <
\frac1q < \frac{2n-1}{2}(\frac12-\frac1r)$, that is,
$\frac{2d}{2d-1} < \alpha < 2$.

\subsection{Preliminary decomposition}
By using the refined Strichartz estimate \eqref{ref-str}, we extract
frequencies and scaling parameters  to get a preliminary
decomposition as follows.

\begin{prop}\label{prop-decom1}
Let $(u_n)_{n\geq1}$ be a sequence of complex valued functions with
$\|u_n\|_{L^2} \leq 1$. Then for any $\delta > 0$, there exists $N =
N(\delta)$, $\rho^j_n \in (0,\infty)$ and  $(f^j_n)_{1 \leq j \leq
N, n \geq 1} \subset L^2$ such that
$$u_n = \sum^N_{j=1} f_n^j + q^N_n$$
with the following properties:
\begin{enumerate}
\item[$1)$] there exist compact set $K = K(N)$ in an annulus $\{\xi : R_1 < |\xi| < R_2\}$
satisfying that
$$(\rho^j_n)^{d/2}|\widehat{f^j_n}(\rho^j_n \xi)| \leq C_\delta \chi_K(\xi)\;\;\mbox{for every}\;\; 1 \leq j \leq N,$$
\item[$2)$] $\limsup_{n \rightarrow \infty}(\frac{\rho^j_n}{\rho^k_n}+\frac{\rho^k_n}{\rho^j_n}) =
 \infty,\;\;\mbox{ if}\;\; j \neq
k,$
\item[$3)$] $\limsup_{n \rightarrow \infty} \|U(\cdot)q^N_n\|_{L_t^qL_x^r} \leq \delta \;\;\mbox{for any}\;\;N \geq
1,$
\item[$4)$] $\limsup_{n \rightarrow \infty}(\|u_n\|_2^2 - (\sum_{j=1}^N\|f^j_n\|^2_2 + \|q^N_n\|^2_2)) = 0.$
\end{enumerate}
\end{prop}

\begin{proof}
Suppose that $\limsup_{n \to
\infty}\|U(\cdot)u_n\|_{L_t^qL_x^r} \le \delta$, then there is
nothing to prove.  So we assume that$ \|U(\cdot ) u_n
\|_{L_t^qL_x^r} > \delta$ for all $n \ge 1$. By refined Strichartz
estimates (Lemma \ref{ref-str-prop}), there exists $A^1_n = \{\xi :
\rho^1_n/2 < |\xi| < \rho^1_n \}$ such that
$$c_1(\rho^1_n)^{d(\frac 1p - \frac 12)}\delta^{\frac 1\theta} \leq
\|\widehat{u_n^1}\|_p\;\;\mbox{for some constant}\;\;c_1,$$ where
$\widehat{u_n^1} = \widehat{u_n}\chi_{A_n^1}$. And for any $\lambda
> 0$
\begin{align*}
\int_{\{|\widehat{u_n^1}| > \lambda \}} |\widehat{u_n^1}|^p d\xi &=
\int_{\{|\widehat{u_n^1}| > \lambda\}}
(\lambda^{2-p}|\widehat{u_n^1}|^p)\lambda^{p-2}d\xi \le
\lambda^{p-2}.
\end{align*}
Thus we have
$$
\Big(\int_{\{|\widehat{u_n^1}| > \lambda\}} |\widehat{u_n^1}|^p d\xi
\Big)^{\frac 1p} \le \lambda^{1-\frac 2p}.
$$

Let $\lambda = (\frac{c_1}2)^{\frac p{2-p}}(\rho^1_n)^{-\frac d2}
\delta^{\frac 1\theta \cdot \frac p{p-2}}$. Then
\begin{align*}
\frac{c_1}2(\rho^1_n)^{d(\frac 1p - \frac 12)}\delta^{\frac 1\theta}
&\leq \Big(\int_{\{ |\widehat{u^1_n}| < \lambda\}}
|\widehat{u^1_n}|^p\Big)^{\frac 1p} \leq
(\omega_d^\frac1d\rho^1_n)^{d(\frac 1p - \frac
12)}\Big(\int_{\{|\widehat{u^1_n}| < \lambda\}}
|\widehat{u^1_n}|^2\Big)^{\frac 12},
\end{align*}
where $\omega_d$ is the measure of unit sphere, which implies that
$$
\frac {c_1'}2\delta^{\frac 1\theta} \leq
\Big(\int_{\{|\widehat{u^1_n}| < \lambda\}}
|\widehat{u^1_n}|^2\Big)^{\frac 12}\qquad (c_1' =
c_1\omega_d^{1/2-1/p}).
$$

Now define $G_n^1(\psi)(\xi)$ by $(\rho_n^1)^{d/2}
\psi(\rho_n^1\xi)$ for measurable function $\psi$. Then by letting
$\widehat{v^1_n} = \widehat{u^1_n}\; \chi_{\{|\widehat{u^1_n}| <
\lambda\}}$ we get $\|v^1_n\|_2 \geq \frac 12 c_1'\delta^{\frac
1\theta}$ and $|G_n^1(\widehat{v^1_n}(\xi))| = (\rho^1_n)^{\frac d2}
\widehat{v^1_n}(\rho^1_n \xi) \leq C_\delta \chi_{A_{1/2, 1}}(\xi)$,
where $A_{R_1, R_2}$ is the annulus $\{\xi : R_1 < |\xi| < R_2\}$.
We can repeat above progress with $u_n - v^1_n$ replacing $u_n$.
After $N ( = N(\delta))$ steps\footnote{At each step, the $L^2$ norm
decreases by at least $\frac12 c_1'\delta^{\frac 1\theta}$.}, we get
$(v^j_n)_{1 \leq j \leq N}$ and $(\rho^j_n)$ such that
\begin{align*}
u_n &= \sum_{j=1}^N v^j_n + q^N_n,\\
\|u_n\|^2_2 &= \sum_{j=1}^N \|v^i_n\|^2_2 + \|q^N_n\|^2_2,\\
\|&U(t)q^N_n\|_{L^r_{t,x}} \leq \delta.
\end{align*}
The second identity follows from disjointness of
\footnote{Actually, we can make them mutually disjoint at each
step.} Fourier supports of $v_n^j$ and $q_n^N$. The third inequality
gives Property 3).

We say $\rho^j_n \perp \rho^k_n$ if and only if $\limsup (\frac
{\rho^k_n}{\rho^j_n} + \frac {\rho^j_n}{\rho^k_n}) = \infty$. Define
$f^1_n$ to be the sum of those $v^j_n$ whose $\rho^j_n$ are not
orthogonal to $\rho^1_n$. Take least $j_0 \in [2,N]$ such that
$\rho^{j_0}_n$ is orthogonal to $\rho^1_n$ and define $f^2_n$ to
be the sum of $v^j_n$ whose $\rho^j_n$ are orthogonal to $\rho_n^1$
but not to $\rho_n^{j_0}$. After finite step, we have $(f^j_n)_{1
\leq j \leq N}$ satisfying properties 2) and 4) because the
Fourier supports are disjoint.

Now we have only to check Property 1). We only consider $f_n^1$. The other cases can be treated similarly.
Since $v^j_n$ collected in $f^1_n$ has $\rho^j_n$ which is not
orthogonal to $\rho^1_n$, we have
\begin{align}\label{non-ortho}
\limsup_{n \rightarrow \infty} \left(\frac {\rho^j_n}{\rho^1_n} +
\frac {\rho^1_n}{\rho^k_n}\right) < \infty. \end{align} And by
construction, we also have $|G^j_n(\widehat{v^j_n})| \leq
C_\delta \chi_{A_{1/2,1}}$. Here $G^j_n(\psi)(\xi) = (\rho_n^j)^\frac d2 \psi(\rho_n^j \xi)$. Since
$G^1_n(\widehat{v^j_n}) = G^1_n(G^j_n)^{-1}G^j_n(\widehat{v^j_n})$
and
$$G^1_n(G^j_n)^{-1}\psi(\xi) = \left(\frac
{\rho^1_n}{\rho^j_n}\right)^{\frac d2}\psi\left(\frac
{\rho^1_n}{\rho^j_n} \xi\right),$$ from the non-orthogonality
\eqref{non-ortho} it follows that there exist $R_1$ and $R_2$ with
$0 < R_1 < R_2$ such that $|G^1_n(\widehat{v^j_n})| \leq
\widetilde{C_\delta} \chi_{A_{R_1, R_2}}$ for all $v_n^j$ collected
in $f_n^1$. This completes the proof of Proposition
\ref{prop-decom1}.
\end{proof}

The next step is devoted to further decomposition of $f_n^j$ to get
time parameters.
\begin{prop}\label{further-decomp}
Suppose that $\{f_n\} \subset L^2$ satisfies
$(\rho_n)^{d/2}|\widehat{f_n}(\rho_n\xi)| \leq \widehat{F}(\xi)$
and $\widehat{F} \in L^\infty(K)$ for some compact set $K
\subset A = \{\xi : 0 < R_1 < |\xi| < R_2\}$. Then there exist
a family $(s^\ell_n)_{\ell \geq 1} \subset \mathbb{R}$ and a
sequence $(\phi^\ell)_{\ell \geq 1} \subset L^2$ satisfying the
following properties:
\begin{enumerate}
\item[$1)$] for $\ell \neq \ell'$
$$\limsup_{n \rightarrow \infty} |s^\ell_n - s^{\ell'}_n| = \infty,$$
\item[$2)$]
for every
$M \geq 1$, there exists $e^M_n \in L^2$ such that $$f_n(x) =
\sum_{\ell=1}^M (\rho_n)^{d/2}(U(s^\ell_n) \phi^\ell)(\rho_n
x) + e^M_n(x)$$ and $$\limsup_{\substack{ M \to \infty \\ n\to \infty  }}  \|U(\cdot)
e^M_n\|_{L_t^qL_x^r} = 0,$$
\item[$3)$]
for any $M \geq 1$,
$$\limsup_{n \rightarrow \infty}\left(\|f_n\|^2 - (\sum_{\ell=1}^M
\|\phi^\ell\|_2^2 + \|e^M_n\|_2^2)\right) = 0.$$
\end{enumerate}
\end{prop}

\begin{proof}
Let  us denote by $\mathcal F$ the collection of functions
$\{F_n\}_{n \geq 1}$ which are given by $\widehat{F_n}(\xi) =
(\rho_n)^{d/2}\widehat{f_n}(\rho_n \xi)$, and define
$$\mathcal{W}(\mathcal F) = \{\text{\small weak-lim}\ U(-s^1_n)F_n \text{ in } L^2 : s^1_n \in \mathbb{R} \}$$
and $\mu(\mathcal F) = \sup_{\phi \in \mathcal{W}(\mathcal F)} \|\phi\|_{L^2}$. Then $\mu(\mathcal F) \leq \limsup_{n \rightarrow \infty} \|F_n\|_{L^2}$.

We may assume that $\mu(\mathcal F) > 0$. As a matter of fact, if
$\mu(\mathcal F) = 0$, we are done because we will show later
\eqref{bound-mu} for some $\theta$,  $0 < \theta < 1$.

Let us choose a subsequence $\{F_n\}$, $s^1_n$ and $\phi^1$ such that $U(-s^1_n)F_n \rightharpoonup \phi^1$ as $n \rightarrow \infty$ and $\|\phi^1\| \geq \frac 12\mu(\mathcal F)$. Let $F^1_n = F_n - U(s^1_n)\phi^1$ and $\mathcal F^1 = \{F^1_n\}$.
Then
\begin{align*}
\limsup_{n \rightarrow \infty}&\|F^1_n\|^2_2\\ &= \limsup_{n \rightarrow \infty} \langle F_n - U(s^1_n)\phi^1,F_n - U(s^1_n)\phi^1\rangle\\
&= \limsup_{n \rightarrow \infty}\langle U(-s^1_n)F_n - \phi,U(-s^1_n)F_n - \phi^1\rangle\\
&= \limsup_{n \rightarrow \infty}\left(\langle F_n, F_n \rangle - \langle U(-s^1_n)F_n, \phi^1\rangle - \langle \phi^1,U(-s^1_n)F_n \rangle + \langle \phi^1,\phi^1 \rangle\right)\\
&= \limsup_{n \rightarrow \infty}\|F_n\|^2_2 - \|\phi^1\|^2_2.
\end{align*}
Repeat the process with $F^1_n$ to get $s_n^2, \phi^2,F^2_n$ and so on.
By taking a diagonal sequence we may write
$$F_n(x) = \sum_{\ell=1}^M U(s_n^\ell)\phi^\ell + F^M_n,$$
which satisfies that $\limsup_{n \rightarrow \infty} \|F_n\|_2^2 = \sum_{\ell=1}^M \|\phi^\ell\|_2^2 + \limsup_{n \rightarrow \infty} \|F_n^M\|_2^2$. So $\sum_{\ell=1}^M \|\phi^\ell\|_2^2$ is convergent, which implies $\limsup_{\ell \rightarrow \infty}\|\phi^\ell\|_2 = 0$. Since  $\mu(\mathcal F^M) \leq 2 \|\phi^{M+1}\|_2$, we get $\limsup_{M\rightarrow\infty} \mu(\mathcal F^M) = 0$.

Now let us define $e_n^M$ by $\widehat F_n^M = \rho_n^\frac d2 \widehat{e_n^M}$. Then the remaining thing is to show
\begin{align}\label{bound-mu}
\limsup_{n \rightarrow \infty} \| U(\cdot) e^M_n \|_{L_t^qL_x^r} \lesssim \mu(\mathcal F^M)^\theta
\end{align}
for some $\theta$ with $0 < \theta < 1$.  By construction, we may assume $\widehat{\phi^\ell}_{1\leq \ell \leq M}$ has common compact support $K$.
Invoking that the pair $(q, r)$ is $\alpha$-admissible with $\frac1q < \frac{2d-1}{2}(\frac12 - \frac1r)$, we get
\begin{align*}
\|U(\cdot)e^M_n\|_{L_t^qL_x^r} = \|U(\cdot)F^M_n \|_{L_t^qL_x^r} \leq \|U(\cdot)F^M_n \|_{L_t^{\widetilde q}L_x^{\widetilde r}}^{\widetilde q/q}\|U(\cdot)F^M_n
\|_{L^\infty_{t,x}}^{1-\widetilde q/q}
\end{align*}
for some $\frac{2d}{2d-1}$-admissible pair $(\widetilde q, \widetilde r)$ with $\frac{\widetilde q}{q} = \frac{\widetilde r}{r}$.
Concerning the first term, from Lemma \ref{str-radial} we have
$$
\|U(\cdot)F^M_n \|_{L_t^{\widetilde q}L_x^{\widetilde r}} \lesssim  R_1^{\frac d2 - \frac{2d}{\widetilde q (2d-a)} - \frac{d}{\widetilde r}} \|F_n^M\|_{L^2} \lesssim R_1^{\frac d2 - \frac{2d}{\widetilde q (2d-a)} - \frac{d}{\widetilde r}}.
$$
Thus for \eqref{bound-mu} it suffices to show $\limsup_{n
\rightarrow \infty} \|U(t)F_n^M\|_{L^\infty_{t,x}} \lesssim
\mu(\mathcal F^M)$. For this we may assume that there exists
$\delta
> 0$ such that $$\limsup_{\substack{ M \to \infty \\ n\to \infty  } }\|U(t)F^M_n\|_{L^\infty_{t,x}} >
\delta.$$ Let $(s^M_n, y^M_n)$ be such that $\|U(t)
F^M_n\|_{L^\infty_{t,x}} = |U(s^M_n)(F^M_n)(y^M_n)|$. Then we show
that $|y^M_n|$ is uniformly bounded.

Let us first observe that for any $x_0, x_1 \in \mathbb R^d$
\begin{align*}
&\qquad |U(t)F^M_n(x_1) - U(t)F^M_n(x_0)| \leq \sup_{x}|\nabla(U(t)F^M_n(x))||x_1 - x_0|\\
&\leq \int
|\xi||e^{it|\xi|^\alpha}\widehat{P^M_n}(\xi)|d\xi|x_1-x_0|
\lesssim (\int_0^{R_2} r^2\cdot r^{n-1} dr)^{\frac 12}\|F^M_n\|_{2}|x_1 - x_0|\\
&\lesssim R_2^{\frac {n+2}2}|x_1 - x_0|.
\end{align*}
From this, we deduce that $|U(s^M_n) F^M_n(y)| > \frac \delta2$ if $|y-y^M_n|
\leq c\frac \delta2$ for some small constant $c > 0$. Since $U(s^M_n)(F^M_n)(y)$ is radially
symmetric, $$|U(s^M_n)(F^M_n)(y)| > \frac \delta2\;\;\mbox{if}\;\; |y^M_n| -
c\frac \delta2 < |y| < |y^M_n| + c\frac \delta2.$$
Taking $L^2$ norm on $|y^M_n| -
c\frac \delta2 < |y| < |y^M_n| + c\frac \delta2$, we have $\frac \delta2 |y^M_n|^{n-1}  \frac \delta2 c \leq \|F^M_n\|_2 \leq 1$, which implies
$|y^M_n|$ is uniformly bounded. Since $|y^M_n|$ is uniformly bounded, there exists $y^M_0$ such that
$y^M_n \rightarrow y^M_0$ as $n \rightarrow \infty$ for some
subsequence. Then for large $n$, $|U(s^M_n)(F^M_n)(y^M_0)|
\geq \frac 12 |U(s^M_n)(F^M_n)(y^M_n)|$. Let $\psi \in
C^\infty_0(\mathbb{R}^d)$ be radially symmetric and such that $\psi
= 1$ on $K$. And let $\psi^M$ be Schwartz function such that $\widehat{\psi^M} =
\psi \widehat{\delta_{y^M_0}}$, where $\delta_{y_0^M}$ is Dirac-delta measure. Then
\begin{align*}
\limsup_{n \rightarrow \infty} \|U(t) F^M_n\|_{L^\infty_{t,x}} &\lesssim \limsup_{n \rightarrow \infty} |U(s^M_n)(F^M_n)(y^M_0)|\\
&= \limsup_{n \rightarrow \infty} |\int U(s^M_n)(F^M_n)(y)\psi^M(y)dy|\\
&\leq \|\psi^M\|_2\mu(\mathcal F^M) \lesssim \mu(\mathcal F^M).
\end{align*}
This completes the proof of Proposition \ref{further-decomp}.
\end{proof}

We now begin the proof of Theorem \ref{main}.

\subsection{Proof of Theorem \ref{main}}\label{s4}  Let us start with the
preliminary decomposition. From Propositions \ref{prop-decom1} and
\ref{further-decomp} we have
\begin{align}\label{pre-decomp}
u_n = \sum_{j = 1}^N \sum_{\ell = 1}^{M_j}\Phi_n^{\ell, j}  + \omega_n^{N,M_1,\cdots,M_N},
\end{align}
where \begin{align*}
&\qquad\qquad\qquad\Phi_n^{\ell, j} = U(t^{\ell,j}_n)[(h^j_n)^{-d/2}\phi^{\ell,j}(\cdot/h^j_n)],\\
&(h^j_n, t^{\ell,j}_n) = ((\rho^j_n)^{-1},(\rho^j_n)^{-\alpha}s^{\ell,j}_n),\quad \omega_n^{N,M_1,\cdots,M_N} = \sum_{j=1}^N e^{j, M_j}_n + q^N_n.
\end{align*}
Then the decomposition satisfies
\begin{enumerate}
\item by constructions, the family $(h^j_n,t^{\ell, j}_n)$ is pairwise orthogonal,
\item the asymptotic orthogonality is satisfied as follows:
\begin{align*}
\|u_n\|_2^2 
= \sum_{j=1}^N\sum_{\ell=1}^{M_j}\|\phi^{\ell,j}\|_2^2 + \|\omega^{N,M_1,\cdots,M_N}_n\|_2^2 + o_n(1)
\end{align*}
and $\|\omega^{N,M_1,\cdots,M_N}_n\|_2^2 = \sum_{j=1}^N \|e^{j,M_j}_n\|_2^2 + \|q^N_n\|_2^2 $ due to disjoint Fourier supports.
\end{enumerate}
We will show that
$U(t)\,\omega_n^{N,M_1,\cdots,M_N}$ converges to zero in a Strichartz norm, i.e.,
\begin{align}\label{str-err}\limsup_{n \rightarrow \infty} \|U(t)\, \omega_n^{N,M_1,\cdots,M_N}\|_{L_t^qL_x^r} \rightarrow 0 \text{ as }\min\{N, M_1,\cdots, M_N\} \rightarrow \infty, \end{align}
where $(q, r)$ is an $\alpha$-admissible pair with $\frac{2d}{2d-1} < \alpha < 2$.
We enumerate the pair $(j, \al)$ by $\upsilon$ satisfying
\begin{center}
$\upsilon(j,\al) < \upsilon(k,\beta)$ if $j + \al < k + \beta$ or $j + \al = k + \beta$ and $j<k$.
\end{center}
After relabeling,
$$u_n = \sum_{1 \leq j \leq l}U(t^j_n)[(h^j_n)^{-d/2}\phi^j(\cdot/h^j_n)] + \omega^l_n$$
where $\omega^j_n = u_M^{N,M_1,\cdots,M_n}$ with $l = \sum_{j=1}^N M_j.$ Then the proof is completed by \eqref{str-err}.

Now let us prove \eqref{str-err}. Given $\varepsilon > 0$, we take a positive number $\Lambda$ such that for every $N \ge \Lambda$,
$$\limsup_{n \rightarrow \infty} \|U(t)\, q^N_n \|_{L_t^qL_x^r} \leq \epsilon/3$$
Then for every $N \geq \Lambda$, we can find $\Lambda_N$ such that whenever $M_j \geq \Lambda_N$,
$$\limsup_{n \rightarrow \infty} \|U(t)\, e^{j,M_j}_n \|_{L_t^qL_x^r} \leq \epsilon/3N.$$
Now we rewrite $\omega_n^{N, M_1, \cdots, M_N}$ by
$$\omega_n^{N,M_1,\cdots,M_N} = q^M_n + \sum_{1 \leq j \leq N} e^{j, M_j \vee \Lambda_N}_n + R^{N, M_1, \cdots, M_n}_n,$$
where $M_j \vee \Lambda_N$ denotes $\max \{M_j, \Lambda_N\}$ and
\begin{align*}
R^{N,M_1,\cdots,M_N}_n = \sum_{1 \leq j \leq N} (e^{j,M_j}_n - e^{j,\Lambda_N}_n) = \sum_{\tiny \begin{array}{c} 1 \leq j \leq N\\ M_j < \Lambda_N\end{array}} \sum_{M_j < \ell < \Lambda_N}\Phi_n^{\ell, j}.
\end{align*}
Then we have
$$\lim_{n \rightarrow \infty} \|U(t) \omega_n^{N,M_1, \cdots, M_N}\|_{L_t^qL_x^r} \leq \frac {2\varepsilon}3 + \lim_{n \rightarrow \infty} \|U(t)R^{N,M_1,\cdots,M_N}_n\|_{L_t^qL_x^r}.$$
In order to handle last term, we need the following lemma which will be proved at the end of this section.
\begin{lem}\label{ortho}
For every $N, M_1, \cdots, M_N$, we have
\begin{align}\label{orth nonlinear1}\limsup_{n \rightarrow \infty} \| \sum_{j=1}^N\sum_{\ell = 1}^{M_j}U(t)\Phi_n^{\ell, j}\|^2_{L_t^qL_x^r} \le \sum_{j=1}^N\sum_{\ell = 1}^{M_j}\limsup_{n \to \infty}\left\|U(t)\Phi_n^{\ell, j}\right\|^2_{L_t^qL_x^r}.\end{align}
\end{lem}
From Lemma \ref{ortho} and Strichartz estimates (Lemma \ref{str-radial}) it follows that
\begin{align*}
\limsup_{n \rightarrow \infty}\|U(t)R_n^{N,M1,\cdots,M_N}\|_{L_t^qL_x^r}^2 &\le \sum_{\tiny \begin{array}{c} 1 \leq j \leq N\\ M_j < \Lambda_N\end{array}} \sum_{M_j < \ell < \Lambda_N} \limsup_{n
\rightarrow \infty}
\|U(t)\Phi_n^{\ell, j}\|_{L_t^qL_x^r}^2\\
&\lesssim \sum_{1 \leq j \leq N} \sum_{\ell > M_j} \|\phi^{\ell,j}\|_2^2.
\end{align*}
Since $\sum_{j, \ell} \|\phi^{\ell,j}\|_2^2$ is convergent,
$$\limsup_{n \rightarrow \infty} \left(\sum_{j=1}^N\sum_{\ell > M_j} \|U(t)\Phi_n^{\ell, j}\|_{L_t^qL_x^r}^2\right)^{\frac 12} \leq \frac \varepsilon3,$$
provided that $\min(N, M_1, \cdots, M_N\}$ is large enough. This completes the proof of Theorem \ref{main}.

\begin{proof}[Proof of Lemma \ref{ortho}]
It suffices to show that for $(j, \ell) \neq (k, \ell')$,
\begin{align}\label{orth nonlinear2}\limsup_{n \rightarrow \infty} \|U(t)\Phi_n^{\ell, j}\; U(t)\Phi_n^{\ell', k}\|_{L_t^\frac q2L_x^\frac r2} = 0.\end{align}
When $(j, \ell) \neq (k, \ell')$, there are two possibilities:
\begin{enumerate}
\item $\limsup_{n \rightarrow \infty} \left(\frac {h^k_n}{h^j_n} + \frac {h^j_n}{h^k_n}\right) = \infty$,
\item $(h^j_n)=(h^k_n)$ and $\limsup_{n\rightarrow\infty}\frac{|t^{\ell, j}_n - t^{\ell', k}_n|}{(h^j_n)^\alpha} = \infty$.
\end{enumerate}
More generally we will prove that if $\Psi_1, \Psi_2 \in L^q_tL^r_x$, then
$$\limsup_{n \rightarrow \infty} \Big\|\frac 1{(h^j_n)^{\frac d2}} \Psi_1(\frac {t - t^{\ell,j}_n}{(h^j_n)^\al}, \frac x{h^j_n}) \frac 1{(h^k_n)^{\frac d2}}\Psi_2(\frac {t-t^{\ell',k}_n}{(h^k_n)^\al},\frac x{h^k_n})\Big\|_{L^{\frac q2}_tL^{\frac r2}_x} = 0.$$
By density argument, it suffices  to show this for $\Psi_1,
\Psi_2 \in C^\infty_0(\mathbb{R} \times \mathbb{R}^d)$. Using
H\"{o}lder's inequality and scaling on space, we have
\begin{align*}
& A_n := \Big\|\frac 1{(h^j_n)^{\frac d2}} \Psi_1(\frac {t - t^{\ell, j}_n}{(h^j_n)^\al}, \frac x{h^j_n})
\frac 1{(h^k_n)^{\frac d2}}\Psi_2(\frac {t-t^{\ell', k}_n}{(h^k_n)^\al},\frac x{h^k_n})\Big\|_{L^{\frac q2}_tL^{\frac r2}_x}
\\
& \leq \Big\|\frac 1{(h^j_n)^{\frac d2 - \frac dr}}\Big\| \Psi_1(\frac {t -
t^{\ell, j}_n}{(h^j_n)^\al}, x)\Big\|_{L^r_x} \frac 1{(h^k_n)^{\frac d2 -
\frac dr}}\Big\|\Psi_2(\frac {t-t^{\ell', k}_n}{(h^k_n)^\al},
x)\Big\|_{L^r_x}\Big\|_{L^{\frac q2}_t}.
\end{align*}
Then by time translation and scaling on time, we have
\begin{align*}
A_n
&\leq \Big\|(\frac {h^j_n}{h^k_n})^{\frac \al q}\Big\|\Psi_1(t, x)\Big\|_{L^r_x}
\Big\| \Psi_2((\frac {h^j_n}{h^k_n})^\al t - \frac{t^{\ell',
k}_n-t^{\ell, j}_n}{(h^k_n)^\al},x)\Big\|_{L^r_x}\Big\|_{L^{\frac q2}_t}.
\end{align*}
Since the support in time of $\|\Psi_1(t, \cdot)\|_{L_x^r}$ is compact, from the above conditions (1) and (2) it readily follows that $\limsup_{n \to \infty} A_n = 0$.
This completes the proof of Lemma \ref{ortho}.
\end{proof}

\section{Nonlinear profile decomposition} In this section we prove
Proposition ~\ref{nlpf} by making use of Theorem \ref{main}.

\medskip

For simplicity of notations, we denote $\Ga^j_n\phi^j$ by $\phi^j_n$ and $\Ga^j_n\psi^j$ by $\psi^j_n$.
First, we will show the forward implication. Fix $I =[a,b]\subset I_n$ for all $n$. We set $e^l_n = u_n - \sum_{j=1}^l \Ga^j_n \psi^j - U(\cdot)\omega^l_n$, and 
$$|\!|\!|e^l_n|\!|\!|_{[I]} := \|e^l_n\|_{C_tL^2_x(I \times \mathbb{R}^d)} + \|e^l_n\|_{L^{q_\circ}_tL^{r_\circ}_x(I \times \mathbb{R}^d)}.$$
Since $ \lim_l \limsup_{n} \|U(\cdot)\omega^l_n\|_{L^{q_\circ}_tL^{r_\circ}_x(I \times \mathbb{R}^d)} =0$ and $$ \|u_n\|_{L^{q_\circ}_tL^{r_\circ}_x(I_n\times\mathbb{R}^d)} \le \sum_{j=1}^l \|\psi^j_n\|_{L^{q_\circ}_tL^{r_\circ}_x(I_n\times\mathbb{R}^d)} +1 $$ for a large $l$, it suffices to show
\begin{equation}\label{beta}
\lim_{l\to \infty} \limsup_{n \rightarrow \infty} |\!|\!|e^l_n|\!|\!|_{[I_n]} \rightarrow 0. \end{equation}

We write the equation for $e^l_n$ in the following:
\begin{align*}
\left\{\begin{array}{l} i(e^l_n)_t + (-\Delta)^\frac\al2 e^l_n = F(\sum_{j=1}^l \psi^j_n + U(\cdot)\omega^l_n + e^l_n) - \sum_{j=1}^l F(\psi^j_n), \\
e^l_n(0, x) = \sum_{j=1}^l \phi^j_n(x) - \psi^j_n(0,x),\end{array}\right.
\end{align*}
where $F(v) = (|x|^{-\al} * |v|^2)v$.
Then Strichartz estimates give
\begin{align}\begin{aligned}\label{err-n}
|\!|\!|e^l_n|\!|\!|_{[I]}& \lesssim \  \|e^l_n(a,\cdot)\|_{L^2_x}\\
&\qquad + \|F(\sum_{j=1}^l\psi^j_n + U(\cdot)\omega^l_n + e^l_n)-F(\sum_{j=1}^l \psi^j_n + U(\cdot)\omega^l_n)\|_{L^{1}_tL^{2}_x(I\times \mathbb{R}^d)} \\
&\qquad+ \|F(\sum_{j=1}^l\psi^j_n + U(\cdot)\omega^l_n) - \sum_{j=1}^l F(\psi^j_n)\|_{L^{1}_tL^{2}_x(I\times \mathbb{R}^d)}.
\end{aligned}\end{align}
To estimate each term on the right hand side, we use the orthogonality of nonlinear profile, in addition to the Hardy-Littlewood-Sobolev inequality. Denote the third term in \eqref{err-n} by $$\beta^l_n := \|F(\sum_{j=1}^l\psi^j_n + U(\cdot)\omega^l_n) - \sum_{j=1}^l F(\psi^j_n)\|_{L^1_tL^2_x(I\times \mathbb{R}^d)}.$$
\begin{lem}\label{estbln}
There exist $n_0, l_0$ such that for all $n \ge n_0,l\ge l_0$,
\begin{align}\label{uni-bound}
\sup_{l,n}\|\sum_{j=1}^l\psi^j_n + U(\cdot)\omega^l_n\|_{L^{q_\circ}_tL^{r_\circ}_x(I \times \mathbb{R}^d)} < \infty,
\end{align}
and
\begin{align}\label{lim}
\lim_{l\to \infty} \limsup_{n \rightarrow \infty} \beta^l_n \rightarrow 0.
 \end{align}
\end{lem}
\begin{proof}
First, we show \eqref{uni-bound}. Since $\lim_{l\to \infty} \limsup_{n \rightarrow \infty}\|U(\cdot)\omega^l_n\|_{L^{q_\circ}_tL^{r_\circ}_x(I \times \mathbb{R}^d)} = 0$, it suffices to show $$\|\sum_{j=1}^l\psi^j_n\|_{L^{q_\circ}_tL^{r_\circ}_x(I \times \mathbb{R}^d)} < \infty.$$
It follows from the small data global well-posedness that
$\sum_{j=l_0}^\infty \| \psi^j_n \|_{L^{q_\circ}_tL^{r_\circ}_x(I \times \mathbb{R}^d)}^2 \leq \sum_{j=l_0}^\infty \|\phi_n^j\|_{L^2_x}^2 < 1.$
for some large $l_0$.
Due to the orthogonality \eqref{orth nonlinear1} and \eqref{orth nonlinear2}, for any $l$, we have
\begin{align*}
\|\sum_{j=1}^l \psi^j_n \|_{L^{q_\circ}_tL^{r_\circ}_x(I \times \mathbb{R}^d)}^2 &\leq \sum_{j=1}^l\|\psi^j_n\|_{L^{q_\circ}_tL^{r_\circ}_x(I \times \mathbb{R}^d)}^2 + o_n(1)  \\
       &\leq \sum_{j=1}^{l_0}\|\psi^j_n\|_{L^{q_\circ}_tL^{r_\circ}_x(I \times \mathbb{R}^d)}^2+2
       \lesssim M\cdot l_0 + 2,
\end{align*}
where $ M$ is a uniform bound of $\{\|u_n^0\|^2_{L^2}\} $.
For \eqref{lim}, we expand the cubical expressions $(F(\sum_l \cdot\,))$ and estimate
\begin{align*}
\beta^l_n 
&\;\;\leq \sum_{\{j_1 = j_2 = j_3 \}^c} \|(|x|^{-\alpha}*(\psi^{j_1}_n\psi^{j_2}_n))\psi^{j_3}_n\|_{L_t^1L_x^2(I \times \mathbb{R}^d)}\\
&\qquad + \sum_{j_1,j_2} \|(|x|^{-\alpha}*(\psi^{j_1}_nU(\cdot)\omega^l_n)\psi^{j_2}_n\|_{L_t^1L_x^2(I \times \mathbb{R}^d)}  \\
&\qquad +  \sum_{j_1, j_2} \|(|x|^{-\alpha}*(\psi^{j_1}_nU(\cdot)\omega^l_n)\psi^{j_2}_n\|_{L_t^1L_x^2(I \times \mathbb{R}^d)} + \|F(U(\cdot)\omega^l_n)\|_{L_t^1L_x^2(I \times \mathbb{R}^d)}.
\end{align*}
Since $(q_\circ, r_\circ)$ is $\alpha$-admissible, we can use the estimate $$\|(|x|^{-\alpha}* (v_1v_2))v_3\|_{L_t^1L_x^2(I \times \mathbb{R}^d)} \lesssim \prod_{1 \le i \le 3}\|v_i\|_{L^{q_\circ}_tL^{r_\circ}_x(I \times \mathbb{R}^d)},$$  to get
\begin{align*}
\beta^l_n &\lesssim \sum_{\{j_1 = j_2 = j_3 \}^c}\prod_{1 \le i \le 3} \|\psi^{j_i}_n\|_{L^{q_\circ}_tL^{r_\circ}_x(I \times \mathbb{R}^d)}\\
&\quad + \|U(\cdot)\omega^l_n\|_{L^{q_\circ}_tL^{r_\circ}_x(I \times \mathbb{R}^d)}\left(\sum_{j_1,j_2} \prod_{i = 1,2}\|\psi^{j_i}_n\|_{L^{q_\circ}_tL^{r_\circ}_x(I \times \mathbb{R}^d)} + \|U(\cdot)\omega^l_n\|_{L^{q_\circ}_tL^{r_\circ}_x(I \times \mathbb{R}^d)}^2\right).
\end{align*}
Then from the orthogonality of nonlinear profiles (as like the proof of Lemma \ref{ortho}), \eqref{uni-bound} and
$$ \limsup\limits_{n \to \infty}\|U(t)\omega_n^l\|_{L^{q_\circ}_tL^{r_\circ}_x} = 0, \quad \text{for each } l $$
we conclude \eqref{lim}.
\end{proof}
In order to handle the second term of \eqref{err-n}, we first use H\"{o}lder's and the Hardy-Littlewood-Sobolev inequality to estimate
\begin{align*}
& \|F(\sum_{j=1}^l \psi^j_n + U(\cdot)\omega^l_n + e^l_n) - F(\sum_{j=1}^l \psi^j_n + U(\cdot)\omega^l_n)\|_{L^1_tL^2_x(I \times \mathbb{R}^d)}\\
&\lesssim \sum_{k=1}^{2} \|\sum_{j=1}^l \psi^j_n + U(\cdot)\omega^l_n\|_{L^{q_\circ}_tL^{r_\circ}_x(I \times \mathbb{R}^d)}^{3-k} \|e^l_n\|_{L^{q_\circ}_tL^{r_\circ}_x(I \times \mathbb{R}^d)}^k.
\end{align*}

Substituting this into \eqref{err-n} and taking limsup, by Lemma \ref{uni-bound} and \eqref{uni-bound} we obtain
\begin{align*}
\limsup_{n\to \infty}|\!|\!| e^l_n |\!|\!|_{[I]} \lesssim \limsup_{n\to \infty}\|e^l_n(a,\cdot)\|_{L^2} & + \limsup_{n\to \infty}\sum_{j=1}^l\| \psi^j_n \|_{L^{q_\circ}_tL^{r_\circ}_x(I \times \mathbb{R}^d)}^{2}\|e^l_n\|_{L^{q_\circ}_tL^{r_\circ}_x(I \times \mathbb{R}^d)}\\
& + \limsup_{n\to \infty}\|e^l_n\|_{L^{q_\circ}_tL^{r_\circ}_x(I \times \mathbb{R}^d)}^2.
\end{align*}
To handle remaining terms in the right hand side, we will divide interval $I_n$ as in following lemma.
\begin{lem}\label{spit}
For given $\epsilon >0 $, there exist intervals $I_n^1, \dots,
I_n^\ell$ such that $I_n$ = $\cup_{i=1}^\ell I^i_n$ and
$$\limsup_{n \to \infty}\sum_{j=1}^l \| \psi^j_n \|_{L^{q_\circ}_tL^{r_\circ}_x(I^i_n \times \mathbb{R}^d)} \leq \epsilon, \quad 1 \le i \le \ell.$$
\end{lem}
\begin{proof}
The global well-posedness for small data and orthogonality give
$$\limsup_{n \to \infty}\|\sum_{j \geq \widetilde{l}} \psi^j_n \|_{L^{q_\circ}_tL^{r_\circ}_x(I^i_n \times \mathbb{R}^d)}^2 \le \limsup_{n \to \infty}\sum_{j \geq \widetilde{l}} \|\psi^j_n\|_{L^{q_\circ}_tL^{r_\circ}_x(I^i_n \times \mathbb{R}^d)}^2 \leq \frac{\epsilon}2$$
for sufficiently large $\widetilde{l}$.
Let $I^1$ be maximal existence interval of $\psi^1$.
Since
$$
\|\psi^1_n\|_{L^{q_\circ}_tL^{r_\circ}_x(I_n \times \mathbb{R}^d)}
= \|\psi^1\|_{L^{q_\circ}_tL^{r_\circ}_x((I_n + t^1_n)*(h^1_n)^\alpha ) \times \mathbb{R}^d)},
$$
there exists $\widetilde{I^1} \subset I^1$ such that
$\|\psi^1\|_{L^{q_\circ}_tL^{r_\circ}_x(\widetilde{I}^1\times \mathbb{R}^d)} <
\infty$  and $(I_n + t^1_n)*(h^1_n)^\alpha  \subset
\widetilde{I^1}.$ Hence we can find $\ell_1$ and $\widetilde{I^1_i}$
such that $\widetilde{I_1} = \cup_{i=1}^{\ell_1} \widetilde{I^1_i}$
and $\|\psi^1\|_{L^{q_\circ}_tL^{r_\circ}_x(\widetilde{I}^1_i \times \mathbb{R}^d)}
\leq {\epsilon}/{2\widetilde{l}}\;.$ Thus
$$
\|\psi^1_n\|_{L^{q_\circ}_tL^{r_\circ}_x(\widetilde{I}^1_{n,i} \times \mathbb{R}^d)} <  {\epsilon}/{2\widetilde{l}}\;,
$$
where $\widetilde{I}^1_{n,i} = I^1_i/(h^1_n)^\alpha - t^1_n$. By
repeating this argument we get $\ell_j$ and $\widetilde{I}^j_{n,i}$,
for $1 \leq j \leq \widetilde{l}$, satisfying
$$
\|\psi^j_n\|_{L^{q_\circ}_tL^{r_\circ}_x(\widetilde{I}^j_{n,i} \times
\mathbb{R}^d)} <  {\epsilon}/{2\widetilde{l}}\;.
$$ Then by taking
intersection of $\widetilde{I}^j_{n,i}$ and $I_n$, we have
$\{I_n^i\}_{i=1}^\ell$ with $\ell = \sum_{i=1}^{\widetilde{l}}
\ell^i$.
\end{proof}
For $I = I_n^1$ we thus have up to a subsequence
\begin{align*}
|\!|\!| e^l_n |\!|\!|_{[I^1_n]}  \lesssim \|e^l_n(0,\cdot)\|_{L^2} +
\beta_n^l + \epsilon^{2}\|e^l_n\|_{L^{q_\circ}_tL^{r_\circ}_x(I^1_n \times \mathbb{R}^d)} +
|\!|\!| e^l_n |\!|\!|_{[I^1_n]}^2,\end{align*} provided $n$ is
sufficiently large. By taking small $\epsilon>0$ we get
$$|\!|\!| e^l_n |\!|\!|_{[I^1_n]} \lesssim \|e^l_n(0,\cdot)\|_{L^2} + \beta_n^l + |\!|\!| e^l_n |\!|\!|_{[I^1_n]}^2.$$
Since $\lim_l\limsup_{n}\|e^l_n(0,\cdot)\|_{L^2} = 0$,
by continuity argument $\limsup_{l, n \rightarrow \infty}|\!|\!| e^l_n |\!|\!|_{[I^1_n]}$ $ = 0.$
Particularly, this implies that
$\limsup_{l, n \rightarrow \infty} \|e^l_n(b^1_n, \cdot)\|_{L^2_x} = 0,$
where $I^1_n = [a^1_n,b^1_n]$ and $a_n^1 = 0$.
Then repeated arguments give
$\lim_{n \rightarrow \infty} |\!|\!| e^l_n |\!|\!|_{[I^j_n]} \rightarrow 0 \text{ as } l \rightarrow \infty$
for $1 \leq j \leq \ell$.

Now we show the implication $(2)\to (1)$. Suppose that the statement
is wrong. Then $\limsup_{n \rightarrow \infty}
\|u_n\|_{L^{q_\circ}_tL^{r_\circ}_x(I_n \times \mathbb{R}^d)} < \infty$ and there
exists $j_0$ such that
\[
\limsup_{n \rightarrow \infty} \|\psi^{j_0}_n\|_{L^{q_\circ}_tL^{r_\circ}_x(I_n
\times \mathbb{R}^d)} = \infty.\] By continuity, for given $M$, we
have $\tilde{I_n} \subset I_n$ satisfying
\begin{align*}
&M < \limsup_{n \rightarrow \infty}
\|\psi^{j_0}_n\|_{L^{q_\circ}_tL^{r_\circ}_x(\tilde{I_n} \times \mathbb{R}^d)},\\
&\limsup_{l \rightarrow \infty}\limsup_{n \rightarrow \infty}
\sum_{j=1}^{l} \|\psi^j_n\|_{L^{q_\circ}_tL^{r_\circ}_x(\tilde{I_n} \times
\mathbb{R}^d)} <\infty.
\end{align*}
Then the implication $(1)\to (2)$ gives $u_n = \sum_{j=1}^l \psi^j_n
+ U(\cdot)\omega^l_n + e^l_n.$ Squaring this, we get
$$|u_n - U(\cdot)\omega^l_n - e^l_n|^2 - {\rm Re}\sum_{j_1 > j_2}^l \psi^{j_1}_n\overline{\psi^{j_2}_n} = \sum_{j=1}^l |\psi^j_n|^2.$$
Then Minkowski's inequality with $q,r \geq 2$ gives
\begin{align*}
& \big\|\big(\sum_{j=1}^l |\psi^j_n|^2\big)^{\frac 12}
\big\|_{L^{q_\circ}_tL^{r_\circ}_x(\tilde{I_n} \times \mathbb{R}^d)}^2
= \big\| \sum_{j=1}^l |\psi^j_n|^2 \big\|_{L^{q_\circ/2}_tL^{r_\circ/2}_x(\tilde{I_n} \times \mathbb{R}^d)}\\
&\qquad\qquad= \big\||u_n - U(\cdot)\omega^l_n - e^l_n|^2 - {\rm Re}\sum_{j_1 > j_2}^l\psi^{j_1}_n\overline{\psi^{j_2}_n}\big\|_{L^{q_\circ/2}_tL^{r_\circ/2}_x(\tilde{I_n} \times \mathbb{R}^d)}\\
&\qquad\qquad \lesssim \|u_n\|_{L^{q_\circ}_tL^{r_\circ}_x(\tilde{I_n} \times \mathbb{R}^d)}^2 + \|U(\cdot)\omega^l_n\|_{L^{q_\circ}_tL^{r_\circ}_x(\tilde{I_n} \times \mathbb{R}^d)}^2 + \|e^l_n\|_{L^{q_\circ}_tL^{r_\circ}_x(\tilde{I_n} \times \mathbb{R}^d)}^2\\
&\qquad\qquad\qquad + \sum_{j_1 > j_2}^l\|\psi^{j_1}_n\overline{\psi^{j_2}_n}\|_{L^{q_\circ/2}_tL^{r_\circ/2}_x(\tilde{I_n} \times \mathbb{R}^d)}.
\end{align*}
Due to orthogonality, we obtain
\begin{align*}
 \limsup_{l \rightarrow \infty}\limsup_{n \rightarrow \infty} \big\|(\sum_{j=1}^l |\psi^j_n|^2)^{\frac 12} \big\|_{L^{q_\circ}_tL^{r_\circ}_x(\tilde{I_n} \times \mathbb{R}^d)}\lesssim \limsup_{n \rightarrow \infty}\|u_n\|_{L^{q_\circ}_tL^{r_\circ}_x(\tilde{I_n} \times \mathbb{R}^d)}.
\end{align*}

On the other hand, we have
\begin{align*}
& \big\||u_n - U(\cdot)\omega^l_n - e^l_n|^2\big\|_{L^{q_\circ/2}_tL^{r_\circ/2}_x(\tilde{I_n} \times \mathbb{R}^d)}\\
&\leq \big\| \sum_{j=1}^l |\psi^j_n|^2 \big\|_{L^{q_\circ/2}_tL^{r_\circ/2}_x(\tilde{I_n} \times \mathbb{R}^d)} + \big\|{\rm Re}\sum_{j_1 > j_2}^l\psi^{j_1}_n\overline{\psi^{j_2}_n} \big\|_{L^{q_\circ/2}_tL^{r_\circ/2}_x(\tilde{I_n} \times \mathbb{R}^d)}.
\end{align*}
And we also have
\begin{align*}
& \big\||u_n - U(\cdot)\omega^l_n - e^l_n|^2\big\|_{L^{q_\circ/2}_tL^{r_\circ/2}_x(\tilde{I_n} \times \mathbb{R}^d)}^{\frac 12}
 = \|u_n - U(\cdot)\omega^l_n - e^l_n\|_{L^{q_\circ}_tL^{r_\circ}_x(\tilde{I_n} \times \mathbb{R}^d)}\\
&\qquad\geq \|u_n\|_{L^{q_\circ}_tL^{r_\circ}_x(\tilde{I_n} \times \mathbb{R}^d)} - \|U(\cdot)\omega^l_n\|_{L^{q_\circ}_tL^{r_\circ}_x(\tilde{I_n} \times \mathbb{R}^d)} - \|e^l_n\|_{L^{q_\circ}_tL^{r_\circ}_x(\tilde{I_n} \times \mathbb{R}^d)}.
\end{align*}
Hence we obtain
\begin{align*}
& \|u_n\|_{L^{q_\circ}_tL^{r_\circ}_x(\tilde{I_n} \times \mathbb{R}^d)} \\
&\leq \big(\big\|\sum_{j=1}^l |\psi^j_n|^2 \big\|_{L^{q_\circ/2}_tL^{r_\circ/2}_x(\tilde{I_n} \times \mathbb{R}^d)} + \big\|{\rm Re}\sum_{j_1 > j_2}^l\psi^{j_1}_n\overline{\psi^{j_2}_n} \big\|_{L^{q_\circ/2}_tL^{r_\circ/2}_x(\tilde{I_n} \times \mathbb{R}^d)}\big)^{\frac 12}\\
&\qquad + \|U(\cdot)\omega^l_n\|_{L^{q_\circ}_tL^{r_\circ}_x(\tilde{I_n} \times \mathbb{R}^d)} + \|e^l_n\|_{L^{q_\circ}_tL^{r_\circ}_x(\tilde{I_n} \times \mathbb{R}^d)}
\end{align*}
and
$
\limsup\limits_{n \to \infty}
\|u_n\|_{L^{q_\circ}_tL^{r_\circ}_x(\tilde{I_n} \times \mathbb{R}^d)}
$
$ \leq
\limsup\limits_{l \to \infty,\,n
\to \infty}$ $\big\|(\sum_{j=1}^l |\psi^j_n|^2)^\frac
12\big\|_{L^{q_\circ}_tL^{r_\circ}_x(\tilde{I_n} \times \mathbb{R}^d)}$ by orthogonality. So it
follows that
$$\limsup_{n \rightarrow \infty} \|u_n\|_{L^{q_\circ}_tL^{r_\circ}_x(\tilde{I_n} \times \mathbb{R}^d)} \approx \limsup_{l \rightarrow \infty}\limsup_{n \rightarrow \infty}\|(\sum_{j=1}^l |\psi^j_n|^2)^\frac 12\|_{L^{q_\circ}_tL^{r_\circ}_x(\tilde{I_n} \times \mathbb{R}^d)}.$$
Therefore, we get
\begin{align*}
M^2 &< \limsup_{n \rightarrow \infty}
\|\psi^{j_0}_n\|_{L^{q_\circ}_tL^{r_\circ}_x(\tilde{I_n} \times \mathbb{R}^d)}^2
\leq \limsup_{l \rightarrow \infty}\limsup_{n \rightarrow \infty}
\|(\sum_{j=1}^l
 |\psi^j_n|^2)^{\frac 12} \|_{L^{q_\circ}_tL^{r_\circ}_x(\tilde{I_n} \times \mathbb{R}^d)}^2\\
&\lesssim \limsup_{n \rightarrow
\infty}\|u_n\|_{L^{q_\circ}_tL^{r_\circ}_x(\tilde{I_n} \times \mathbb{R}^d)}^2
 \leq
\limsup_{n \rightarrow \infty}\|u_n\|_{L^{q_\circ}_tL^{r_\circ}_x(I_n \times
\mathbb{R}^d)}^2,
\end{align*}
which gives a contradiction by letting $M\to \infty$. This
completes the proof.


\section{Blowup Phenomena }
In this section we provide the proofs of Theorem ~\ref{th:minimal
blow up},~\ref{cpbs}, and Corollary~\ref{mass concentration}

\subsection*{Proof of Theorem~\ref{th:minimal blow up}} By definition
of $\delta_0$, there exist blowup solutions $\{u_n\}_{n=1}^\infty$
with initial data $\{u_{0,n}\}_{n=1}^\infty \subset L^2_x$ such that
$\|u_{0,n}\| \searrow \delta_0$ as $n \rightarrow \infty$. By using
time translation and scaling symmetry, we may assume that
$$\|u_n\|_{L^{q_\circ}_tL^{r_\circ}_x([0,1]\times \mathbb{R}^d)} \rightarrow \infty \text{ as } n \rightarrow \infty.$$

Then we apply Theorem \ref{main} to  $\{u_{0,n}\}$ to get linear profiles $\{\phi^j,h^j_n,s^j_n\}$.
From Proposition \ref{nlpf}, we obtain nonlinear profiles $\{\psi^j\}$ associated with $\{\phi^j,h^j_n,s^j_n\}$.

Since $\limsup_{n \rightarrow \infty}\|u_n\|_{L^{q_\circ}_tL^{r_\circ}_x([0,1]\times \mathbb{R}^d)} =  \infty$, Proposition \ref{nlpf} says that there exists $j_0$ such that $\psi^{j_0}$ blows up and so we have $\|\phi^{j_0}\|_{L^2} \geq \delta_0$. And by Theorem \ref{main}, we have
$$\|\phi^{j_0}\|_{L^2}^2 \leq \sum_{j \geq 1} \|\phi^{j}\|_{L^2}^2 \leq \limsup_{n \rightarrow \infty} \|u_{0,n}\|_{L^2}^2 = \delta_0^2$$ which implies
$$\|\psi^{j_0}(0,\cdot)\|_{L^2} = \|\phi^{j_0}\|_{L^2} \leq \delta_0.$$
Hence, $\|\phi^{j_0}\|_{L^2}$ should be $\delta_0$. For the proof of the second conclusion,
we apply the above argument to the sequence $\{ u(t_n)\}$.

\subsection*{Proof of Theorem~\ref{cpbs}} Let $u_n(t,x) = u(t + t_n,
x)$. Then we have $$\int_{\mathbb{R}^d} |u_n|^2 dx =
\int_{\mathbb{R}^d} |u|^2 dx,$$ and $$\limsup_{n \rightarrow \infty}
\|u_n\|_{L^{q_\circ}_tL^{r_\circ}_x([0, T^* - t_n] \times \mathbb{R}^d)} =
\limsup_{n \rightarrow \infty} \|u_n\|_{L^{q_\circ}_tL^{r_\circ}_x([-t_n, 0] \times
\mathbb{R}^d)} = \infty.$$

Let $\{\phi^j,\psi^j,h^j_n,s^j_n\}$ be family of linear and
nonlinear profiles associated with $\{u_n(0,\cdot)\}$ which are
obtained in Theorem \ref{main} and Proposition \ref{nlpf}. We
take the inverse of symmetry group (or we redefine $s^{j}_n :=
-\frac {s^{j}_n}{(h^{j}_n)^\al}$ and $h^{j}_n := \frac 1{h^{j}_n}$).
Then by Proposition \ref{nlpf} for $I_n = [0, T^* - t_n]$, there
exists $j_0$ such that $$\limsup_{n \rightarrow \infty}
\|\psi^{j_0}\|_{L^{q_\circ}_tL^{r_\circ}_x(I^{j_0}_n \times \mathbb
R^d)} = \infty,$$ where $I^{j_0}_n := [s_n^{j_0}, (T^* -
t_n)/(h_n^{j_0})^\alpha + s^{j_0}_n]$.

Let $s^{j_0} := \limsup_{n \rightarrow \infty} s^{j_0}_n$. From
Lemma \ref{wpasy}, we obtain $s^{j_0} \neq \infty$. Hence $s^{j_0} =
-\infty$, or $s^{j_0} = 0$. If $s^{j_0} = -\infty$, then $\psi^{j_0}$
blows up at $T^{**}$ and $\limsup_{n \rightarrow \infty} (T^* -
t_n)/(h^{j_0}_n)^\alpha \geq T^{**}$. Applying the same
argument to $\tilde{I_n} = [-t_n, 0]$, we get  $\tilde{j_0}$
which satisfies
$$\limsup_{n \rightarrow \infty} \|\psi^{\tilde{j_0}}\|_{L^{q_\circ}_tL^{r_\circ}_x(\tilde{I}^{\tilde{j_0}}_n \times \mathbb R^d)} = \infty,$$
where $\tilde{I}^{\tilde{j_0}}_n := [( - t_n)/(h_n^{\tilde{j_0}})^\alpha + s^{\tilde{j_0}}_n, s_n^{\tilde{j_0}}]$. Since
$\|u(0,x)\|_{L^2} < \sqrt{2} \delta_0$, there cannot be two blowup profiles. Hence $\tilde{j_0}$ should be $j_0$. Therefore, from Lemma \ref{wpasy}, we get $s^{j_0} \neq -\infty$.

Now we have $s^{j_0} = 0$. Then Theorem \ref{main} gives
$$(\Ga^{j_0}_n)^{-1}u_n(0,\cdot) = \phi^{j_0} + \sum_{j \neq j_0}^l (\Ga^{j_0}_n)^{-1}\Ga^j_n\phi^j
+ (\Ga^{j_0}_n)^{-1}\omega^l_n.$$
Due to the orthogonality, $(\Ga^{j_0}_n)^{-1}\Ga^j_n\phi^j
\rightharpoonup 0$ weakly in $L^2$ as $n \rightarrow \infty$. And
since $\limsup_{n \rightarrow \infty} \|\omega^l_n\|_{L^{q_\circ}_tL^{r_\circ}_x}
\rightarrow 0$ as $l \rightarrow \infty$, the uniqueness of weak
limit gives $(\Ga^{j_0}_n)^{-1}\omega^l_n \rightharpoonup 0 \text{
weakly in } L^2.$ Hence we have
$$(\Ga^{j_0}_n)^{-1}u(t_n,\cdot) \rightharpoonup \phi^{j_0}\text{
weakly in } L^2.$$
Therefore, by taking $h_n = h^{j_0}_n$ and $\phi = \phi^{j_0}$, we
see \eqref{weakcon} and \eqref{hhh}. This completes the proof of
Theorem \ref{cpbs}.

\subsection*{Proof of Corollary~\ref{mass concentration}} By Theorem
\ref{cpbs}, there exists $\phi \in L^2_x$ such that $ \|\phi\|_{L^2}
\geq \delta_0,$ and \eqref{weakcon} and \eqref{hhh} hold. Hence we
have  for $ R > 0$,
$$\limsup_{n \rightarrow \infty} (h_n)^{d} \int_{|x| \leq R} |u(t_n,h_n x)|^2 dx \geq \int_{|x| \leq R} |\phi|^2
dx.
$$
After dilation, we get
$$\limsup_{n \rightarrow \infty} \int_{|x| \leq Rh_n} |u(t_n,x)|^2 dx \geq \int_{|x| \leq R} |\phi|^2 dx.$$
Since $\frac {(T^* - t_n)^{1/\alpha}}{\lambda(t_n)} \rightarrow 0$ as $t_n \rightarrow T^*$, we get $\frac {h_n}{\lambda(t_n)} \rightarrow 0$ and
$$\limsup_{n \rightarrow \infty} \int_{|x| \leq \lambda(t_n)} |u(t_n,x)|^2 dx
\geq \int_{|x| \leq R} |\phi|^2 dx .$$ Since $\int |\phi|^2dx \geq
\delta_0^2$, letting $R\to \infty$, we get \eqref{masscon}.

\appendix

\section{}

The local well-posedness of \eqref{eqn} is obtained in \cite{chho}.
The well-posedness for a given asymptotic state is similar and
fairly standard. We provide its proof for completeness.

\begin{lem}\label{wpasy}
Given $g \in L^2(\mathbb{R}^d)$, there exists a positive $T$ and a
unique solution $u$ to \eqref{eqn} such that $u \in
C_tL^2_x([T,\infty) \times \mathbb{R}^d) \cap L^{q_\circ}_tL^{r_\circ}_x([T,\infty)
\times \mathbb{R}^d)$ and
$$\|u(t) - U(t)g\|_{L^2_x} \rightarrow 0 \text { as } t \rightarrow \infty.$$
\end{lem}

\begin{proof}
We sketch the proof as the argument is rather standard. We
define nonlinear mapping $\mathcal{N}$ by
$$\mathcal N(v)(t) := i\lambda\int_t^\infty U(t - s)(|x|^{-\alpha}*|U(s)g + v(s)|^2)(U(s)g + v(s))ds$$
for $v$ in Banach space $X = X_{T,\ep}$ given by
\begin{align*}
X := \{ v \in &C_tL^2_x([T,\infty] \times \mathbb{R}^d) \cap L^{q_\circ}_tL^{r_\circ}_x([T,\infty) \times \mathbb{R}^d) : \\& \|v\|_{C_tL^2_x([T,\infty) \times \mathbb{R}^d)} + \|v\|_{L^{q_\circ}_tL^{r_\circ}_x([T,\infty) \times \mathbb{R}^d)} \leq \ep \}.
\end{align*}
Using the Strichartz estimate (Lemma \ref{str-radial}) and Christ-Kiselev lemma, one can get
\begin{align*}
\|\mathcal N(v)\|_{C_tL^2_x([T,\infty)\times\mathbb{R}^d)} &+ \|\mathcal N(v)\|_{L^{q_\circ}_tL^{r_\circ}_x([T,\infty)\times\mathbb{R}^d)} \lesssim \|U(s)g + v(s)\|^{3}_{L^{q_\circ}_tL^{r_\circ}_x([T,\infty)\times\mathbb{R}^d)}\\
&\lesssim (\|U(s)g\|^{3}_{L^{q_\circ}_tL^{r_\circ}_x([T,\infty)\times\mathbb{R}^d)} + \|v(s)\|^{3}_{L^{q_\circ}_tL^{r_\circ}_x([T,\infty)\times\mathbb{R}^d)}).
\end{align*}
Since
$\|U(s)g\|_{L^{q_\circ}_tL^{r_\circ}_x([T,\infty)\times\mathbb{R}^d)}
\lesssim \|g\|_{L^2_x}$by Lemma \ref{str-radial}, $\mathcal N$
becomes a self-mapping on $X$ for sufficiently large $T$. Similarly
one can easily prove that $\mathcal N$ is a contraction mapping
on $X$. Lastly the absolute continuity gives $\|v(t)\|_{L^2_x}
\rightarrow 0$ as $t \rightarrow \infty$.

Now we write $u(t)$ as
$$u(t) = U(t)g + v(t).$$
Then $\|u(t) - U(t)g\|_{L^2_x} \rightarrow 0 \text{ as } t \rightarrow \infty$.
It remains to show that
\begin{align}\label{sol-u} u(\tau) = U(\tau - t)u(t) -i \lambda \int^\tau_t U(\tau - s)((|x|^{-\alpha} * |u|^2)u)(s)ds.\end{align}
In fact, since $v(\tau) = \mathcal N(v)(\tau)$, one can show that
$$v(\tau) = U(\tau-t)v(t)-i\lambda\int_t^\tau U(\tau-s)(|x|^{-\alpha}*|u|^2)u(s)\,ds.$$
Thus
$$
u(\tau) = U(\tau)g + v(\tau) = U(\tau - t)(U(t)g + v(t)) - i\lambda \int_t^\tau U(\tau-s)(|x|^{-\alpha}*|u|^2)u(s)\,ds,
$$
which yields \eqref{sol-u}.
\end{proof}

\section*{Acknowledgments} Y. Cho and G. Hwang are supported by NRF grant 2011-0005122 (Republic of Korea), S. Lee in part by NRF grant 2012-008373 (Republic of Korea).
S. Kwon is partially supported by TJ Park science fellowship and NRF grant 2010-0024017 (Republic of Korea). \medskip

%

\end{document}